\documentclass{rsproca}
\usepackage{epsfig}

\newtheorem{theorem}{\bf Theorem}[section]

\newtheorem{lemma}{\bf Lemma}[section]

\newtheorem{definition}{\bf Definition}[section]

\newcommand{\R}{{\mathbb R}}

\newcommand{\N}{{\mathbb N}}

\newcommand{\C}{{\mathcal C}}

\newcommand{\eps}{\varepsilon}

\newcommand{\dd}{\;{\rm d}}
\newcommand{\PP}{{\mathcal P}}

\newcommand{\W}{{\mathcal W}}

\newcommand{\un}{{\rm 1\kern -2.5pt l}}

\newcommand\pref[1]{(\ref{#1})}

\let \eps\varepsilon

\DeclareMathOperator{\diam}{diam}

\def\<#1,#2>{\left<#1,#2\right>}

\newcommand\ovx{{\overline{x}}}
\newcommand\oX{{\overline{X}}}

\newcommand\ws{\stackrel{*}{\rightharpoonup}}
\newcommand\oJ{\overline{J}}

\def\PP{{\cal P}}

\newcommand\tiln{{\widetilde{\nu}}}

\begin{document}

\title{From Nash to Cournot-Nash equilibria via the Monge-Kantorovich problem}

\author{
Adrien Blanchet$^{1}$, Guillaume Carlier$^{2}$}

\address{$^{1}$IAST/TSE (GREMAQ, Universit\'e de Toulouse), 21 All\'ee de Brienne, 31015 Toulouse, France\\
$^{2}$CEREMADE - Universit\'e Paris Dauphine, Place de Lattre de Tassigny, 75775 Paris C\'edex 16, France}

\subject{35K65, 35K40, 47J30, 35Q92, 35B33}

\keywords{Nash equilibria, games with a continuum of players, Cournot-Nash equilibria, Monge-Kantorovich optimal transportation problem.}

\corres{Adrien Blanchet\\
\email{Adrien.Blanchet@ut-capitole.fr}}

\begin{abstract}
The notion of Nash equilibria plays a key role in the analysis 
of strategic interactions in the framework of $N$ player games. Analysis 
of Nash equilibria is however a complex issue when the number of players 
is large. In this article we emphasize the role of optimal transport 
theory in: 1) the passage from Nash to Cournot-Nash 
equilibria as the number of players tends to infinity, 2) the analysis 
of Cournot-Nash equilibria.
\end{abstract}




\maketitle
\section{Introduction}

Since the seminal work of Nash \cite{Nash1950}, the notion of Nash equilibrium  for $N$-person games has become the key concept to analyse strategic interaction situations and  plays a major role not only in economics but also in biology, computer science, etc. The existence of Nash equilibria  generally relies on nonconstructive fixed-point argument as in the original work of Nash. They are therefore generally hard to compute especially for a large number of players. A traditional question in game theory and theoretical economics is whether the analysis of equilibria simplifies as the number $N$ of players becomes so large that one can use models with a continuum of agents. Indeed, since  Aumann's seminal works \cite{Aumann,Aumann2} models with a continuum of agents have occupied a distinguished position in economics and game theory.   The challenging new theory of Mean-Field Games  of Lasry and Lions \cite{mfg1,mfg2,mfg3} shed some new light on this issue and renewed the interest in the rigorous derivation of continuum models as the limit of models with $N$ players, letting $N$ tend to infinity. Mean-Field Games theory addresses complex dynamic situations (stochastic differential games with a large number of players) and in the sequel, we will only consider simpler static situations.

\smallskip

 In \cite{Schmeidler}, Schmeidler introduced a notion of non-cooperative equilibrium in games with a continuum of agents, having in mind such diverse applications as elections, many small buyers from a few competing firms, drivers that can
choose among several roads. Mas-Colell \cite{MasColell}  reformulated Schmeidler's analysis in a framework that presents great analogies with the Monge-Kantorovich optimal transport problem. Indeed, in Mas-Colell's formulation, players are characterised by some type (productivity, tastes, wealth, geographic position, etc.), whose probability distribution $\mu$ is given. Each agent has to choose a strategy so as to minimise a cost that not only depends on her type but also on the probability distribution of strategies resulting from the behaviour of the rest of the population. A Cournot-Nash equilibrium then is a joint probability distribution of types and strategies with first marginal $\mu$ (given) and second marginal $\nu$ (unknown) that is concentrated on cost minimising strategies. The difficulty of the problem lies in the fact that the cost depends on $\nu$. For relevant applications, this dependence can be quite complex since there may be both attractive (mimetism) and repulsive (congestion) effects in the cost. 
\smallskip

Another situation where one looks for joint probabilities (or couplings) is the optimal transport problem of Monge-Kantorovich. In the Monge-Kantorovich problem, one is given two probability spaces $(\Theta, \mu)$ and $(X, \nu)$, a transport cost $c$ and one looks for the cheapest way to transport $\mu$ to $\nu$ for the cost $c$. In Monge's original formulation, one looks for a transport  $T$ between $\mu$ and $\nu$ {\it i.e.} a measurable map from $\Theta$ to $X$ such that $T_\#\mu=\nu$ where   $T_\#\mu$ denotes the image of $\mu$ by $T$ {\it i.e.}
\[T_\#\mu(B)=\mu(T^{-1} (B)),\quad  \mbox{$\forall B \subset X$ measurable}\]
that minimises the total transport cost 
$$\int_{\Theta} c(\theta, T(\theta)) \,\mu(\!\dd\theta)\;.$$
Since there may exist no transport map  between $\mu$ and $\nu$ and the mass conservation constraint $T_\#\mu=\nu$ is highly nonlinear, Kantorovich relaxed Monge's formulation in the 1940's by considering the notion of transport plan (that allows mass splitting and therefore convexifies Monge's problem).  A transport plan between $\mu$ and $\nu$ is by definition a joint probability measure on $\Theta\times X$  having $\mu$ and $\nu$ as marginals. Denoting by  $\Pi(\mu, \nu)$  the set of transport plan between $\mu$ and $\nu$, the least cost of transporting $\mu$ to $\nu$ for the cost $c$ then is the value of the Monge-Kantorovich optimal transport problem:
\[
\W_c(\mu, \nu):=\inf_{\gamma\in \Pi(\mu, \nu)}  \int_{\Theta \times X} c(\theta, x)  \,\gamma(\!\dd\theta, \dd x)\;.
\]
In the case where we have as same source and target space, a compact metric space $X$ with distance $d_X$, the previous definition defines the usual $1$-Wasserstein metric, $\W_1$:
\[
\W_1(\mu, \nu):= \inf_{\gamma\in \Pi(\mu, \nu)}  \int_{X \times X} d_X(x, y) \, \gamma(\!\dd x, \dd y)
\]
which is well-known to metrize the weak-$*$ topology. We refer the reader to the  textbooks of Villani \cite{villani,villani2} for a detailed exposition of optimal transport theory. 
\smallskip

The purpose of the present article is twofold:
\begin{itemize}
\item emphasise the role of optimal transport theory as a simple but powerful tool to obtain rigorous passage from Nash to Cournot-Nash as the number of players tends to infinity (see Theorem \ref{mixcnn}),  
\item give an account of~\cite{BC,abgc2} which identifies some classes of games with a continuum of players for which Cournot-Nash equilibria can be characterised thanks to optimal transport and actually computed numerically in some nontrivial situations.  
\end{itemize}
\smallskip

The paper is organised as follows. Section \ref{Nash} recalls the classical results of Nash regarding equilibria for $N$-person games. In Section \ref{CN}, we define Cournot-Nash equilibria in the framework of games with a continuum of agents. In Section \ref{ncn}, we give asymptotic results concerning the convergence of Nash equilibria to Cournot-Nash equilibria as the number of players tends to infinity. In Section \ref{cnmk}, we give an account on recent results from \cite{BC,abgc2} which enable one to find Cournot-Nash equilibria using tools from optimal transport and perform numerical simulations. 
\section{Nash equilibria}\label{Nash}
An $N$-person game, consists of a set of $N$ players $i\in \{1, \cdots, N\}$, each player $i$ has a  strategy space $X_i$, setting
\[\oX:=\Pi_{i=1}^N X_i\] 
and for $x\in X$, setting $x=(x_i, x_{-i})\in X=X_i \times \Pi_{j\neq i } X_j$, the cost function of player $i$ is a function
\[J_i: \; \oX=X_i \times  X_{-i} \to \R \quad \mbox{where} \quad X_{-i}:=\Pi_{j\neq i } X_j.\] 
In this classical setting a Nash-equilibrium is defined as follows:
\begin{definition}[Nash equilibrium]
A \emph{Nash-equilibrium} is a collection of strategies $\ovx=(\ovx_1, \cdots, \ovx_N)=(\ovx_i, \ovx_{-i})\in \oX$ such that, for every $i\in \{1, \cdots, N\}$ and every $x_i \in X_i$, one has:
\[J_i(\ovx_i, \ovx_{-i})\le J_i(x_i, \ovx_{-i}).\]
\end{definition}

The fundamental theorem of Nash states the following existence result for Nash equilibria:

\begin{theorem}[Existence of Nash equilibria]\label{Nash1}
If:
\begin{itemize}
\item $X_i$ is a convex compact subset of some locally convex Hausdorff topological vector space,
\item for every $i\in \{1,\cdots, N\}$, $J_i$ is continuous on $\oX$ and for every $x_{-i}\in X_{-i}$, $J_i(., x_{-i})$ is quasi-convex on $X_i$,
\end{itemize} 
then there exists at least one Nash equilibrium.
\end{theorem}

Since the convexity assumptions in the theorem above are quite demanding (and rule out the case of finite games, {\it i.e.} the case of finite strategy spaces $X_i$), Nash also introduced the mixed strategy extension. A mixed strategy for player $i$ is by definition probability measure $\pi_i\in \PP(X_i)$ and given a profile of mixed strategies $(\pi_1,\cdots \pi_N)\in \Pi_{i=1}^N \PP(X_i)$, the cost for player $i$ reads
\[\oJ_i (\pi_1,\cdots \pi_N):=\int_X J_i(x_1,\cdots x_N) \otimes_{j=1}^N \pi_j(\!\dd x_j).\]  
A Nash equilibrium in mixed strategies then is a Nash-equilibrium for the mixed strategy extension of the game. Since, as soon as the strategy spaces are metric compact spaces, this extension satisfies the continuity and convexity assumptions of Theorem \ref{Nash1}, Nash's second fundamental result states that:

\begin{theorem}[Existence of Nash equilibria in mixed strategies]\label{Nash2}
Any game with compact metric strategy spaces and continuous costs has at least one Nash equilibrium in mixed strategies. In particular, finite games admit Nash equilibria in mixed strategies.
\end{theorem}

\section{Cournot-Nash equilibria}\label{CN}

Since the seminal work  of Aumann \cite{Aumann,Aumann2}, games with a continuum of players have received a lot of attention in game theory and economics. The  setting is the following: agents are characterised by a type $\theta$ belonging to some compact metric space $\Theta$. The type space $\Theta$ is also endowed with a Borel probability measure $\mu \in \PP(\Theta)$ which gives the distribution of types in the agents population. Each agent has to choose a strategy $x$ from some strategy space, $X$, again a compact metric space. The cost of one agent does not depend only on her type and the strategy she chooses but also on the other agents' choice through the probability distribution $\nu\in \PP(X)$ resulting from the whole population strategy choice. In other words, the cost is given by some function $F\in \C(\Theta\times X\times \PP(X))$ where $\PP(X)$ is endowed with the weak-$*$ topology (equivalently, the $1$-Wasserstein metric $\W_1$). In this framework, an equilibrium can conveniently be described by a joint probability measure on $\Theta\times X$ which gives the joint distribution of types and strategies and which is consistent with the cost-minimising behaviour of agents, this leads to the following definition:
\begin{definition}[Cournot-Nash equilibrium]\label{cnedefi}
A \emph{Cournot-Nash} equilibrium for $F$ and $\mu$ is a $\gamma\in \PP(\Theta\times X)$ such that:
\begin{itemize}
\item the first marginal of $\gamma$ is $\mu$: ${\Pi_{\Theta}}_\# \gamma=\mu$, 
\item $\gamma$ gives full mass to cost-minimising strategies:
\[\gamma\left(\left\{(\theta, x) \in \Theta\times X \; : F(\theta, x, \nu)=\min_{y\in X} F(\theta, y, \nu)\right\} \right)=1.\]
where $\nu$ denotes the second marginal of $\gamma$: ${\Pi_{X}}_\# \gamma=\nu$.  
\end{itemize}
A Cournot-Nash equilibrium $\gamma$ is called \emph{pure} if it  is  of the form $\gamma=({\rm{id}} ,T)_{\#}\mu$ for some Borel map $T$ : $\Theta\to X$ (that is agents with the same type use the same strategy).
\end{definition}

By a standard fixed-point argument, see Schmeidler \cite{Schmeidler} or Mas-Colell \cite{MasColell}, one obtains:
\begin{theorem}[Existence of Cournot-Nash equilibria]\label{existcne}
If $F\in \C(\Theta\times X\times \PP(X))$ where $\PP(X)$ is endowed with the weak-$*$ topology, there exists at least one Cournot-Nash equilibrium.
\end{theorem}

The continuity assumption above is actually rather demanding and the aim of Section~\ref{cnmk} is to present some results of \cite{BC,abgc2} for existence and sometimes uniqueness under less demanding regularity assumptions. 
\section{From Nash to Cournot-Nash}\label{ncn}
Our aim now is to explain how Cournot-Nash equilibria can be obtained as limits of Nash equilibria as the number of players tends to infinity. What follows is very much inspired by results from the Mean Field Games theory of Lasry and Lions, see in particular the notes of Cardaliaguet \cite{Carda}, following P.-L. Lions' lectures at Coll\`ege de France. The main improvement with respect to \cite{Carda} is that we deal with a situation with heterogeneous players and therefore cannot reduce the analysis to symmetric equilibria. \smallskip

Let $X$ and $\Theta$ be compact and metric spaces with respective distance $d_X$ and $d_\Theta$ and let $\Theta_N:=\{\theta_1, \cdots, \theta_N\}$ be a finite subset of the type space $\Theta$. Consider an $N$-person game where all the agents have the same strategy space $X$. We assume that the cost of player $i$, depends on her type $\theta_i\in \Theta_N$, her strategy $x_i$ and is symmetric with respect to the other players strategies $x_{-i}$: 
\[
J_i^N(x_i, x_{-i})=J^N(\theta_i, x_i, x_{-i})=J^N(\theta_i, x_i, (x_{\sigma(j)})_{j\neq i}) \quad \mbox{for every $\sigma\in S_{N-1}$, }
\]
where $S_{N-1}$ denotes the set of permutations of $\{1,\cdots,N\}\setminus\{i\}$. We moreover assume that there is a modulus of continuity $\omega$ such that for every for every $N$, every $(\theta_i, \theta_j) \in \Theta_N\times \Theta_N$, every  $(x_i, x_{-i})$ and $(y_i, y_{-i})$ in $X^N$, one has
\begin{multline}
\left\vert J^N(\theta_i, x_i, x_{-i}) -J^N(\theta_j, y_i, y_{-i})\right\vert   \le \omega\left(d_\Theta(\theta_i, \theta_j)\right) +  \omega \left(d_X(x_i, y_i)\right) \\
+\omega\left(\W_1\left( \frac{1}{N-1} \sum_{j\neq i} \delta_{x_j},  \frac{1}{N-1} \sum_{j\neq i} \delta_{y_j}\right)\right) \label{modulxmt}
\end{multline}
For further reference, let us observe, following \cite{Carda} that the $\W_1$ distance above is nothing but the quotient (by permutations) distance:
\begin{equation}\label{wassquotient}
\W_1\left( \frac{1}{N-1} \sum_{j\neq i} \delta_{x_j},  \frac{1}{N-1} \sum_{j\neq i} \delta_{y_j}\right)=\min_{\sigma\in S_{N-1}} \frac{1}{N-1}  \sum_{j\neq i} d_X(x_j, y_{\sigma(j)}).
\end{equation}
As in \cite[Theorem~2.1]{Carda}, under \pref{modulxmt}, one can extend $J^N$ to $\Theta\times X \times \PP(X)$ by the classical Mc~Shane construction {\it i.e.} by defining for every $(\theta,x, \nu)\in \Theta\times X\times \PP(X)$, the cost $F^N(\theta, x, \nu)$ by the formula:
\begin{equation*}
 F^N(\theta, x, \nu)=\inf_{(x_{-i}, \theta_i) \in X^{N-1}\times \Theta_N} \left\{J^N(\theta_i, x, x_{-i})+\omega(d_\Theta(\theta_i, \theta))+ \omega\left(\W_1\left(\nu, \frac{1}{N-1} \sum_{j\neq i}  \delta_{x_j}\right)\right)   \right\}. 
\end{equation*}
The previous assumptions are therefore equivalent to the fact that player $i$'s cost is given by
\[
 J_i^N(x_i, x_{-i})=J^N(\theta_i, x_i, x_{-i})=F^N\left(\theta_i, x_i, \frac{1}{N-1} \sum_{j\neq i} \delta_{x_j}\right)
 \]
with $F^N$ a sequence of uniformly equi-continuous functions on $\Theta\times X\times \PP(X)$.

\begin{theorem}[Pure Nash equilibria converge to Cournot-Nash equilibria]\label{fromntocn}
Let $\ovx^N=(x_1^N,\cdots, x_N^N)$ be a Nash equilibrium for the game above, and define:
\[\mu^N:=\frac{1}{N} \sum_{i=1}^N \delta_{\theta_i}, \quad \nu^N:=\frac{1}{N} \sum_{i=1}^N \delta_{x_i^N}\quad \mbox {and} \quad\gamma^N:=\frac{1}{N} \sum_{i=1}^N \delta_{(\theta_i, x_i^N)}\]
Assume that, up to the extraction of sub-sequences,
\[\mu^N \ws \mu, \quad \nu^N\ws \nu,\quad \gamma^N\ws \gamma\quad \mbox {and} \quad F^N \to F \mbox{ in } \C(\Theta\times X\times \PP(X)) \]
then $\gamma$ is a Cournot-Nash equilibrium for $F$ and $\mu$ in the sense of Definition \ref{cnedefi}.
\end{theorem}
\begin{proof}
First set 
$$\tiln_i^N:=  \frac{1}{N-1} \sum_{j\neq i}  \delta_{x_j^N}\;.$$ 
Since $\ovx^N=(x_1^N,\cdots, x_N^N)$ is a Nash equilibrium we have,
\[\forall y \in X,\quad F^N(\theta_i, x_i^N, \tiln_i^N)\le F^N(\theta_i, y, \tiln_i^N)\;.\]
Hence, using $\W_1(\tiln_i^N, \nu^N)\le \diam(X)/N$, we obtain
\[\forall y \in X,\quad F^N(\theta_i, x_i^N, \nu^N)\le F^N(\theta_i, y, \nu^N)+\eps_N\;,\]
where $\eps_N:=2 \omega( \diam(X)/N)$ tends to $0$ as $N$ goes to $\infty$. Summing over $i$ and dividing by $N$ gives
\[\int_{\Theta\times X} F^N(\theta, x, \nu^N) \;\gamma^N(\!\dd\theta, \dd x) \le  \int_\Theta \min_{y\in X} F^N(\theta, y, \nu^N) \;\mu^N(\!\dd\theta)+\eps_N.\] 
Since $F^N(.,., \nu^N)$ converges in $\C(\Theta\times X)$ to $F(.,., \nu)$, letting $N$ tend to $\infty$ in the previous inequality, we obtain:
\[\int_{\Theta\times X} F(\theta, x, \nu) \;\gamma(\!\dd\theta, \dd x) \le  \int_\Theta \min_{y\in X} F(\theta, y, \nu) \;\mu(\!\dd\theta)\;.\]
Since $\gamma$ obviously has marginals $\mu$ and $\nu$, we readily deduce that $\gamma$ is a Cournot-Nash equilibrium for $F$ and $\mu$.
\end{proof}
As already discussed above, not all the games have a pure Nash equilibrium. So that the previous result might be of limited interest. This is why, we now consider the mixed strategy extension which always has Nash equilibria  as recalled in Section~\ref{Nash}.
\begin{theorem}[Nash equilibria in mixed strategies converge to Cournot-Nash equilibria]\label{mixcnn}
If Assumption \pref{modulxmt} holds for $\omega(t)=Kt$, then the conclusions of Theorem~\ref{fromntocn} applies to the extension in mixed strategies with the extended cost
\[\oJ^N(\theta_i, \pi_i, \pi_{-i}):=\int_{X^N} J^N(\theta_i, x, x_{-i})\; \pi_i(\!\dd x) \otimes_{j\neq i} \pi_j(\!\dd x_j)\;.\]
\end{theorem}
This theorem is a direct consequence of the fact that for uniformly Lipschitz $J^N$, the extension in mixed strategies still satisfies \pref{modulxmt}: 
\begin{lemma} Under the Assumptions of Theorem~\ref{mixcnn}, the extended cost $ \oJ^N$ satisfies
  \begin{multline*}
\left\vert \oJ^N(\theta_i, \pi_i, \pi_{-i}) -\oJ^N(\theta_j, \eta_i, \eta_{-i})\right\vert \le  K\,d_\Theta\left(\theta_i, \theta_j\right) +  K \, \W_1\left(\pi_i, \eta_i\right)\\
+K \min_{\sigma\in S_{N-1}} \frac{1}{N-1}  \sum_{j=1, \; j\neq i}^N  \W_1(\pi_j, \eta_{\sigma(j)}) \;.
\end{multline*}
for every $N \in \N$, $(\theta_i, \theta_j) \in \Theta_N^2$, $(\pi_i, \pi_{-i})$ and $(\eta_i, \eta_{-i})$ in $\PP(X)^N$.
\end{lemma}

\begin{proof}
The inequality 
$$\left\vert \oJ^N(\theta_i, \pi_i, \pi_{-i}) -\oJ^N(\theta_j, \pi_i, \pi_{-i})\right\vert \le  K\,d_\Theta\left(\theta_i, \theta_j\right)$$
is obvious by integration. Let $\Gamma\in \Pi(\pi_i, \eta_i)$ be such that 
$$\int_{X\times X} d_X(x,y) \,\Gamma(\!\dd x, \dd y)=\W_1(\pi_i, \eta_i)\;.$$
We have:
\begin{multline*}
\left\vert \oJ^N(\theta_i, \pi_i, \pi_{-i}) -\oJ^N(\theta_i, \eta_i, \pi_{-i})\right\vert \leq \\
 \int_{X^{N-1}}  \left(   \int_{X\times X} \left\vert J^N(\theta_i, x, x_{-i})-J^N(\theta_i, y, x_{-i}) \right\vert \Gamma(\!\dd x, \dd y)\right) \otimes_{j\neq i} \pi_j(\!\dd x_j)\\
\le K \int_{X\times X} d_X(x,y) \,\Gamma(\!\dd x, \dd y)=K \,\W_1(\pi_i, \eta_i).
\end{multline*}
Let $\sigma_\in S_{N-1}$ and for $j\neq i$, let then $\Gamma_j \in \Pi(\pi_j, \eta_{\sigma(j)})$ be such that 
\[\int_{X\times X} d_X(x,y)\, \Gamma_j(\!\dd x, \dd y)=\W_1(\pi_j, \eta_{\sigma(j)}).\]
Then define $\Gamma \in \PP(X^{N-1}\times X^{N-1})$ by
\[\int_{X^{N-1}\times X^{N-1}} \varphi \Gamma:=\int_{X^{N-1}\times X^{N-1}} \varphi(x_{-i}, y_{-i}) \otimes_{j\neq i} \Gamma_j(\!\dd x_j, \dd y_j)\;.\]
We then have by symmetry, definition of $\Gamma$ and \pref{modulxmt}:
\begin{eqnarray*}
\left\vert \oJ^N(\theta_i, \pi_i, \pi_{-i})-\oJ^N (\theta_i, \pi_i, \eta_{-i}) \right\vert &=&\left\vert \oJ^N(\theta_i, \pi_i, \pi_{-i})-\oJ^N (\theta_i, \pi_i, {\eta_{\sigma(j)}}_{j\neq i} )\right\vert\\
&=&\left\vert \int (J^N( \theta_i, x, x_{-i})- J^N( \theta_i, x, y_{-i}) ) \,\pi_i(\!\dd x) \,\Gamma(\!\dd x_{-i}, \dd y_{-i}) \right\vert \\
&\le& \frac{K}{N-1} \sum_{j\neq i}  \int_{X\times X}  d_X(x_j, y_j) \,\Gamma_j(\!\dd x_j, \dd y_j) \\
&=&  \frac{K}{N-1}\sum_{j\neq i} \W_1(\pi_j, \eta_{\sigma(j)})
\end{eqnarray*}
and since $\sigma$ is arbitrary, we have
\[\left\vert \oJ^N(\theta_i, \pi_i, \pi_{-i})-\oJ^N (\theta_i, \pi_i, \eta_{-i}) \right\vert \leq K \min_{\sigma\in S_{N-1}} \frac{1}{N-1}  \sum_{j=1, \; j\neq i}^N  \W_1(\pi_j, \eta_{\sigma(j)}) \;.\]
which ends the proof.
\end{proof}
\section{Solving Cournot-Nash via the Monge-Kantorovich problem}\label{cnmk}
In this final section, we give a brief and informal account of our recent works \cite{BC,abgc2} and the computation of Cournot-Nash equilibria in the separable case {\it i.e.} the case where, using the notations of Section~\ref{CN}:
\begin{equation}\label{sepform}
F(\theta, x, \nu)=c(\theta, x)+ V(\nu, x).
\end{equation}
All the strategic interactions are then captured by the term  $V(\nu, x)$. Typical effects that have to be taken into account are: 
\begin{itemize}
\item \emph{congestion  (or repulsive) effects}: frequently played strategies are costly. This can be captured by a term of the form $x \mapsto f(x, \nu(x))$ where slightly abusing notations $\nu$ is identified with its density with respect to some reference measure $m_0$ according to which congestion is measured. The congestion effect leads to dispersion in the strategy distribution $\nu$,
\item \emph{positive interactions (or attractive effects)} like mimetism: such effects can be captured by a term like $$\int_X \phi(x,y) \,\nu(\!\dd y)$$ with $\phi$ minimal when $x=y$. The positive interaction effect leads to the concentration of the strategy distribution. 
\end{itemize}

In the sequel we shall therefore consider a term $V(\nu,x)$ which reflects both effects. Cournot-Nash equilibria have to balance the attractive and repulsive effects in some sense, but their structure is not easy to guess when these two opposite effects are present. As a benchmark, one can think of $\nu$ being absolutely continuous with respect to some reference measure $m_0$ and a potential of the form
\begin{equation}\label{sepform2}
V(\nu,x)=f(x, \nu(x))+\int_X \phi(x,y) \,\nu(\!\dd y)
\end{equation}
where $f(x,.)$ is increasing and $\phi$ is some continuous interaction kernel. Note that congestion terms, though natural in applications, pose important mathematical difficulties since they are not regular. In particular, in the presence of such terms one cannot use the existence result stated in Theorem~\ref{existcne}. \smallskip

In\cite{BC}, we emphasized that the structure of equilibria, which we take pure to simplify the exposition {\it i.e.} we look for a strategy map $T$ from $\Theta$ to $X$, is as follows: 
\begin{itemize}
\item players of type $\theta$ choose a cost-minimising strategy $T(\theta)$ {\it i.e.} 
$$\forall x \in X,\quad c(\theta, T(\theta))+V(\nu, T(\theta))\le c(\theta, x)+V(\nu, x)\;.$$
$T$ is called \emph{best-reply map}, 
\item $\nu$ is the image of $\mu$ by $T$:  $\nu:=T_\#\mu$. This is a mass conservation condition. 
\end{itemize}

The aim is to compute $\nu$ and $T$ from the previous two conditions. In general, these two highly nonlinear conditions are difficult to handle, but in \cite{BC,abgc2} we identified some cases in which they can be greatly simplified. We then obtained some uniqueness results and/or numerical methods to compute equilibria. Roughly speaking, in an euclidean setting for instance, the \emph{best-reply} condition leads to an optimality condition relating $\nu$ and $T$ whereas the mass conservation condition leads to a Jacobian equation. Combining the two conditions therefore leads to a certain nonlinear partial differential equation of Monge-Amp\`ere type.\smallskip

We shall review some tractable examples in the next paragraphs, but  wish to observe that in the separable setting, Cournot-Nash equilibria are very much related to optimal transport. More precisely, for $\nu\in \PP(X)$, let $\Pi(\mu, \nu)$ denote the set of probability measures on $\Theta \times X$ having $\mu$ and $\nu$ as marginals and let $\W_c(\mu, \nu)$ be the least cost of transporting $\mu$ to $\nu$ for the cost $c$ {\it i.e.} the value of the Monge-Kantorovich optimal transport problem:
\[
\W_c(\mu, \nu):=\inf_{\gamma\in \Pi(\mu, \nu)}  \int_{\Theta \times X} c(\theta, x)  \,\gamma(\!\dd\theta, \dd x).
\]
It is obvious that the optimal transport  problem above admits solutions since the admissible set is convex and weakly-$*$ compact. Let us then denote by $\Pi_o(\mu,\nu)$ the set of optimal transport plans {\it i.e.}
\[
\Pi_o(\mu,\nu):=\left\{\gamma\in \Pi(\mu, \nu) \; : \; \int_{\Theta \times X} c(\theta,x)   \, \gamma(\!\dd\theta, \dd x)=\W_c(\mu, \nu)\right\}\;.
\]

The link between Cournot-Nash equilibria and optimal transport is then based on the following straightforward observation: if $\gamma$ is a Cournot-Nash equilibrium and $\nu$ denotes its second marginal then $\gamma\in \Pi_o(\mu, \nu)$, see \cite[Lemma~2.2]{BC} for details. Hence if $\gamma$ is pure then it is induced by an optimal transport. Optimal transport theory has developed extremely fast in the last decades  and one can therefore naturally take advantage of this powerful theory, for which we refer to the books of Villani \cite{villani,villani2}, to study Cournot-Nash equilibria.

\subsection{Dimension one}

Assume that both $\Theta$ and $X$ are intervals of the real line, say $[0,1]$. The two equations of cost minimisation and mass conservation can be simplified to characterise Cournot-Nash equilibria. This is why under quite general assumptions on $c$, $f$ and $\phi$ we prove in \cite[Section~6.3]{abgc2} that the optimal transport map $T$ associated to a Cournot-Nash equilibrium satisfies a certain nonlinear and non-local differential equation. For instance when $f=\log(\nu)$, the map $T$ corresponding to an equilibrium solves:
\begin{equation}\label{odeforT1}
T'(\theta)=C \,\mu(\theta)\, \exp\left(-\int_0^\theta \partial_\theta c(s, T(s)) \dd s+c(\theta,T(\theta))+\int_{0}^1 \phi\left(T(\theta), T(\beta) \right) \,\mu(\!\dd\beta) \right)
\end{equation}
supplemented with the boundary conditions $T(0)=0$ and  $T(1)=1$.
The case of a power congestion function $f(\nu)=\nu^{\alpha}$ with $\alpha >0$, is more involved since in this case, equilibrium densities may vanish. In this case,  one cannot really write a differential equation for the transport map but working with the generalised inverse of the transport map still gives a tractable equation for equilibria. Hence equilibria can be computed numerically, see \cite{abgc2}, as shown in Figure~\ref{fig:1d}.
\begin{figure}[h!]
 \begin{minipage}[t]{.49\linewidth}
\centering\epsfig{figure=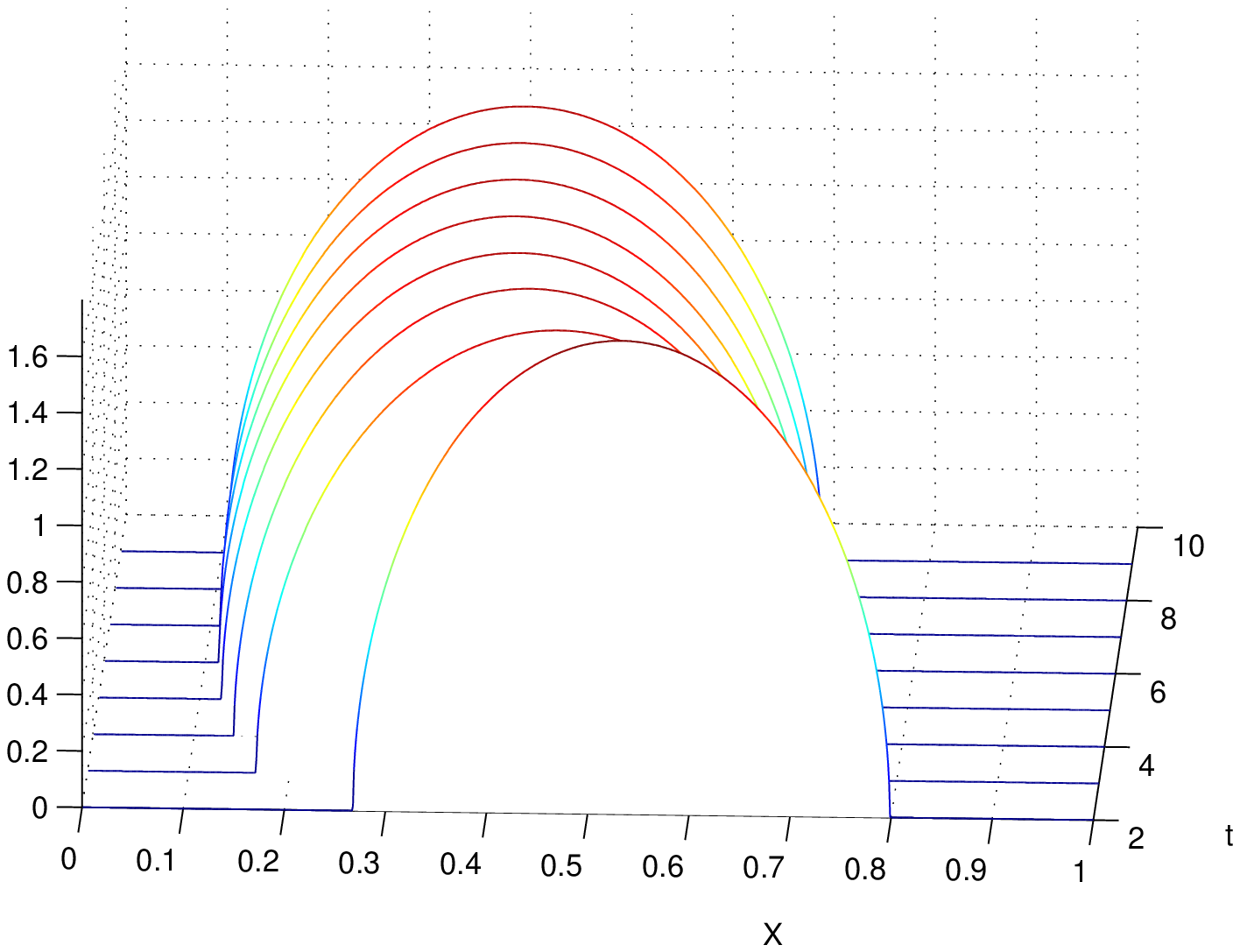,width=6.4cm}
 \end{minipage} \hfill
\begin{minipage}[t]{.49\linewidth}
\centering\epsfig{figure=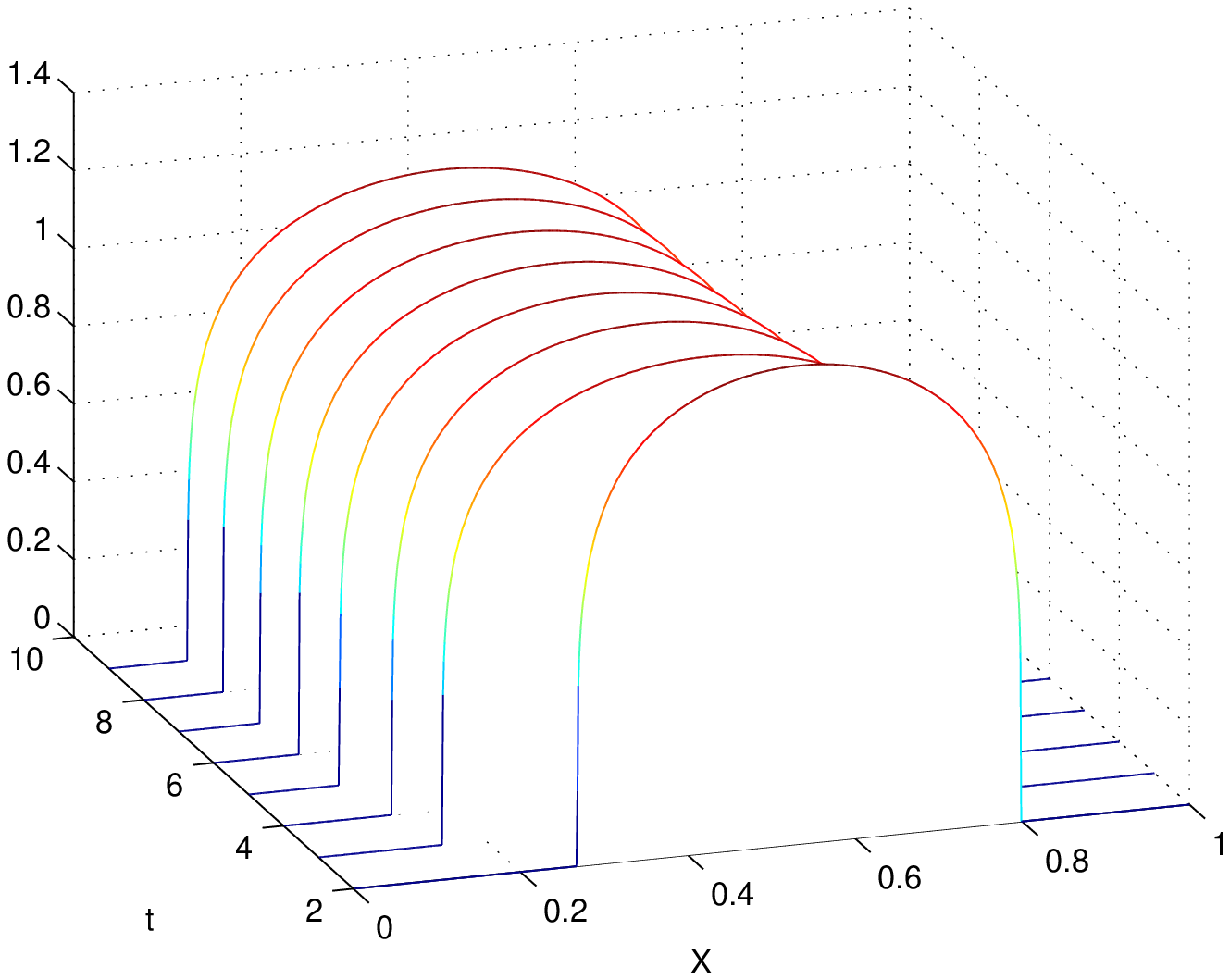,width=6.4cm}
 \end{minipage} 
\caption{\small The convergence to the distribution of actions $\nu$ at equilibrium in the case of $f(\nu)=\nu^{\alpha}$ with $\alpha=2$ on the left and $\alpha=5$ on the right. Here $\phi(x,y):=10(2x-y-0.4)^2$, $c(\theta, x)=|\theta-x|^4/4$ and $\mu$ is uniform on $[0,1]$.}\label{fig:1d}
\end{figure}

Let us mention that in higher dimensions, for a quadratic $c$ and a logarithmic $f$, finding an equilibrium amounts to solve the following Monge-Amp\`ere equation, see \cite[Section~4.3]{BC}:
\begin{multline}\label{mongeampeq}
\mu(\theta)=\det(D^2u(\theta)) \exp\left(-\frac{1}{2} \vert \nabla u(\theta)\vert^2 +\theta\cdot \nabla u(\theta)-u(\theta)\right) \times  \\
\exp\left(-\int_{\Theta} \phi\left(\nabla u(\theta), \nabla u(\theta_1)\right)\dd \mu(\theta_1)\right).
\end{multline}
\subsection{Variational approach}
In \cite{BC}, we observe that whenever $V$ takes the form \pref{sepform2} with a \emph{symmetric}  kernel $\phi$ then it is the first variation of the energy
\begin{equation*}
J(\nu):=\int_X F(x, \nu(x)) \,m_0(\!\dd x)+ \frac{1}{2} \int_{X\times X}  \phi(x,y)\, \nu(\!\dd x)\, \nu(\!\dd y)
\quad \mbox{with} \quad F(x, \nu):=\int_0^\nu f(x, s) \dd s\;.
\end{equation*}
We can thus prove that the condition defining equilibria is in fact the Euler-Lagrange equation for the variational problem
\begin{equation}\label{variationalpb}
\inf_{\nu\in \PP(X)} \left\{ \W_c(\mu, \nu)+J(\nu)\right\}\;.
\end{equation} 
More precisely if $\nu$ solves \pref{variationalpb} and $\gamma\in \Pi_o(\mu, \nu)$ then $\gamma$ actually is a Cournot-Nash equilibrium. Not only this approach gives new existence results but also, and more surprisingly, uniqueness in some cases since the functional above may have hidden convexity properties. In dimension one, under suitable convexity assumptions the variational problem above can be solved numerically in an efficient way as illustrated by Figure~\ref{fig:jko}.  
\begin{figure}[h!]
\centering\epsfig{figure=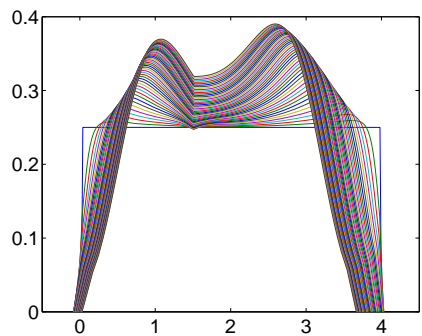,width=6.4cm}
\caption{\small The distribution of actions $\nu$ at equilibrium in the case $\phi(x,y)=(x-1.6)^4/4+(x-y)^2/200$ and $\mu$ is a distribution of support $[0.5,0.6] \cup [3.7,3.8]$. The fact that $\mu$ has a disconnected support results in a corner in the density $\nu$.}\label{fig:jko}
\end{figure}

\subsection{A two-dimensional case by best-reply iteration}

For some specific forms in \pref{sepform} where $c$ and  the potential $V$  have strong convexity  properties with respect to $x$, one can look for equilibria directly by \emph{best-reply iteration}. For a fixed strategy distribution $\nu$, each type $\theta$ has a unique optimal strategy $T(\nu, \theta)$, transporting $\mu$ by this best-reply map then gives a new measure $T(\nu, .)_\# \mu$, an equilibrium corresponds to a fixed-point  {\it i.e.} a solution of $T(\nu, .)_\# \mu=\nu$. For a quadratic $c$, the previous equation simplifies a lot and, in \cite{abgc2},  we obtained conditions under which the best-reply map is a contraction for the Wasserstein metric which guarantees  uniqueness of the equilibrium and convergence of  best-reply iterations. See Figure~\ref{fig:2d} for a two-dimensional example.  
\begin{figure}[h!]
 \begin{minipage}[t]{.49\linewidth}
\centering\epsfig{figure=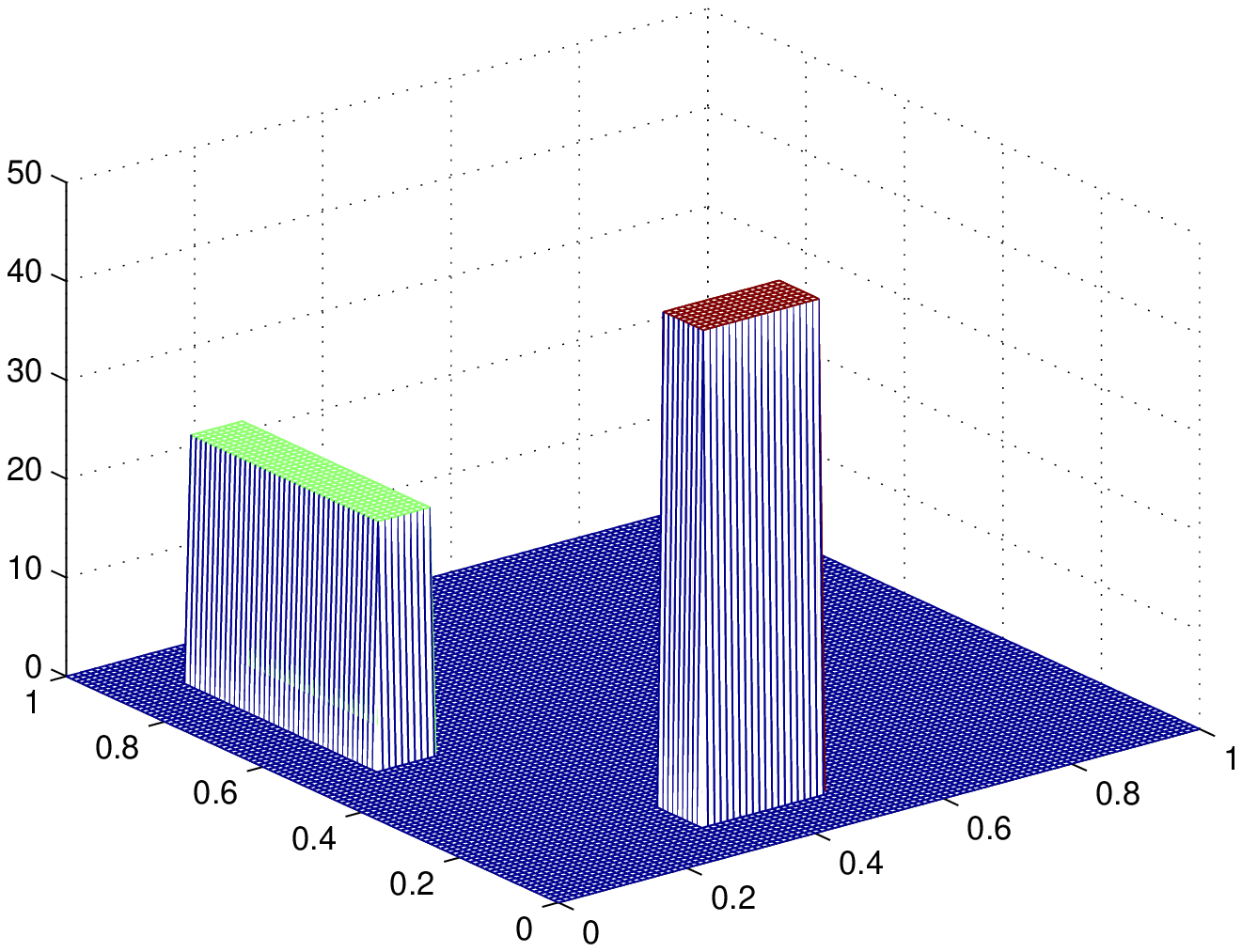,width=6.4cm}
 \end{minipage} \hfill
\begin{minipage}[t]{.49\linewidth}
\centering\epsfig{figure=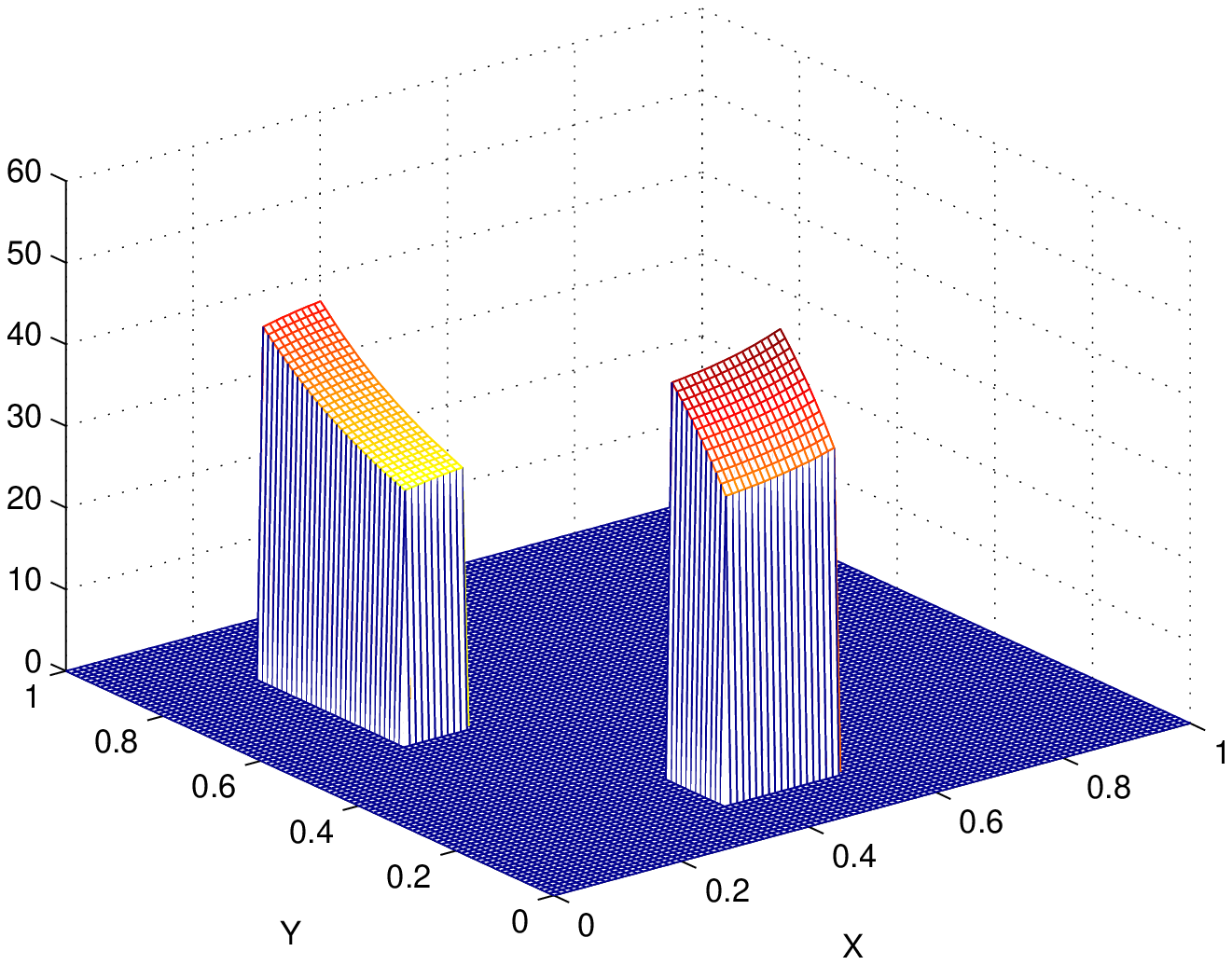,width=6.4cm}
 \end{minipage} 
\caption{\small In the case $\phi(x,y)=|x-(1,1)|^2+|x-y|^4/100$ and $\mu$ is made of two uniform distributions of support $[1/10,2/10]\times[3/10,5/10] \cup [5/10,9/10]\times[1/10,2/10]$: the distribution $\mu$ on the left and the distribution of actions $\nu$ at equilibrium on the right.}\label{fig:2d}
\end{figure}

{\bf{Acknowledgements.}} The authors gratefully acknowledge the support of INRIA and the ANR through the Projects ISOTACE (ANR-12-MONU-0013) and OPTIFORM (ANR-12-BS01-0007).

\bibliographystyle{siam}
\bibliography{biblio}
\end{document}